\documentclass[12pt,leqno]{amsart}
\usepackage{amsmath,amssymb,amsfonts}
\usepackage{eucal,enumerate}
\usepackage{xypic,graphicx,psfrag}
\usepackage{mathrsfs}
\usepackage{hyperref}

\usepackage{color} 

\newtheorem{lem}{Lemma}[section]

\newtheorem{prop}[lem]{Proposition}
\newtheorem{thm}[lem]{Theorem}

\theoremstyle{definition}
\newtheorem{exam}[lem]{Example}

\numberwithin{equation}{section}
\numberwithin{table}{section}
\numberwithin{figure}{section}

\allowdisplaybreaks

\renewcommand{\phi}{\varphi} 
\renewcommand{\epsilon}{\varepsilon}

\newcommand\setm{\setminus}
\newcommand\inv{^{-1}}
\newcommand\transpose{^{\text{T}}}
\newcommand\codim{\operatorname{codim}}
\newcommand\rk{\operatorname{rk}}

\newcommand\cA{\mathscr{A}}		
\newcommand\cB{\mathcal{B}}

\renewcommand\cH{\mathcal{H}}	
\newcommand\cP{\mathcal{P}}
\newcommand\cS{\mathscr{S}}
\newcommand\cU{\mathcal{U}}	

\newcommand\bbR{\mathbb{R}}
\newcommand\bbZ{\mathbb{Z}}
\newcommand\pB{\mathbb B}	

\newcommand\ba{\mathbf a}
\newcommand\bfb{\mathbf b}
\newcommand\bc{\mathbf c}
\newcommand\bd{\mathbf d}
\newcommand\bx{\mathbf x}
\newcommand\by{\mathbf y}
\newcommand\bz{\mathbf z}

\newcommand\Alpha{\mathsf{A}}
\newcommand\Beta{\mathsf{B}}
\newcommand\Eta{\mathsf{H}}
\renewcommand\S{\mathsf{\Sigma}}
\newcommand\G{\mathsf{\Gamma}}
\newcommand\Ps{\mathsf{\Psi}}

\newcommand\Kot{Kot\v{e}\v{s}ovec}
\newcommand\cube{[0,1]^{2q}}

\allowdisplaybreaks

\begin{document}

\newcommand\xrefcomment[1]{\{#1\}}	
\renewcommand\xrefcomment[1]{}		

\newcommand\B{V.5.1\xrefcomment{B}}		
\newcommand\BB{5.1\xrefcomment{BB}}		
\newcommand\Phvdiag{III.3.1\xrefcomment{P:hvdiag}}	

\thispagestyle{empty}

\title{A $q$-Queens Problem \\
VI.  The Bishops' Period 
\\[10pt]  
\today}

\author{Seth Chaiken}
\address{Computer Science Department\\ The University at Albany (SUNY)\\ Albany, NY 12222, U.S.A.}
\email{\tt sdc@cs.albany.edu}

\author{Christopher R.\ H.\ Hanusa}
\address{Department of Mathematics \\ Queens College (CUNY) \\ 65-30 Kissena Blvd. \\ Queens, NY 11367-1597, U.S.A.}
\email{\tt chanusa@qc.cuny.edu}

\author{Thomas Zaslavsky}
\address{Department of Mathematical Sciences\\ Binghamton University (SUNY)\\ Binghamton, NY 13902-6000, U.S.A.}
\email{\tt zaslav@math.binghamton.edu}

\begin{abstract}
The number of ways to place $q$ nonattacking queens, bishops, or similar chess pieces on an $n\times n$ square chessboard is essentially a quasipolynomial function of $n$ (by Part~I of this series).  The period of the quasipolynomial is difficult to settle.  Here we prove that the empirically observed period 2 for three to ten bishops is the exact period for every number of bishops greater than 2.  The proof depends on signed graphs and the Ehrhart theory of inside-out polytopes.
\end{abstract}



\thanks{Version of \today.}
\thanks{The first two authors thank the very hospitable Isaac Newton Institute for facilitating their work on this project.  The third author gratefully acknowledges support from PSC-CUNY Research Awards PSCOOC-40-124, PSCREG-41-303, TRADA-42-115, TRADA-43-127, and TRADA-44-168.}

\maketitle
\pagestyle{myheadings}
\markright{\textsc{A $q$-Queens Problem. VI \quad \today}}\markleft{\textsc{Chaiken, Hanusa, and Zaslavsky}}


\section{Introduction}\label{intro}

The famous $n$-Queens Problem is to count the number of ways to place $n$ nonattacking queens on an $n\times n$ chessboard.  That problem has been solved only for small values of $n$; there is no real hope for a complete solution.  In this series of papers we treat a more general problem wherein we place $q$\label{d:q} identical pieces like the queen or bishop on an $n\times n$ square board and we seek a formula for $u(q;n)$,\label{d:indistattacks} the number of ways to place them so that none attacks another.  The piece may be any one of a large class of traditional and fairy chess pieces called ``riders'', which are distinguished by the fact that their moves have unlimited distance.  
We proved in Part~I \cite{QQs1} that in each such problem the number of solutions, times a factor of $q!$, is a quasipolynomial function of $n$; that is, $q!u(q;n)$ is given by a cyclically repeating sequence of polynomials in $n$ and $q$, the exact polynomial depending on the residue class of $n$ modulo some number $p$\label{d:p} called the \emph{period} of the function; and furthermore that each coefficient of the quasipolynomial is a polynomial function of $q$.  
Here we prove that for three or more bishops the period is always exactly 2.\footnote{This paper was originally announced as Part~V, in Parts~I and II.}  This period was previously observed by \Kot\ for $3 \leq q \leq 10$ as a result of his extensive computations for five to ten bishops, added to previous work by Fabel for three and four bishops (see \cite[third ed., pp.\ 126--129 for $q\leq6$; 6th ed., pp.\ 234--241, 254--257 for $q\leq10$]{ChMath}).

The number of nonattacking placements of $q$ unlabelled bishops on an $n\times n$ board is denoted by $u_\pB(q;n)$.  The number for labelled bishops is therefore $q!u_\pB(q;n)$.

\begin{thm}\label{T:bishopsperiod}
For $q \geq 3$, the quasipolynomial $q!u_\pB(q;n)$ involved in counting the nonattacking positions of $q$ bishops on an $n\times n$ board has period equal to $2$.  For $q<3$ the period is $1$.
\end{thm}

To get our results we treat non-attacking configurations as lattice points $\bz:=(z_1,\ldots,z_q)$, $z_i=(x_i,y_i)$, in a $2q$-dimen\-sion\-al inside-out polytope (see Section~\ref{partI}).  The Ehrhart theory of inside-out polytopes (from \cite{IOP}) implies quasipolynomiality with polynomials of degree $2q$ and that the period divides the least common multiple of the denominators of the coordinates of vertices of the inside-out polytope.  We  find the structure of these coordinates explicitly: in Lemma~\ref{L:bishopsdenom} we show that a vertex of the bishops' inside-out polytope has each $z_i \in \{0,1\}^2$ or $z_i = (\frac12,\frac12)$.  From that, along with a formula from Part~III \cite{QQs3} for the coefficient of $n^{2q-6}$ that implies the period is even if $q\geq3$,  Theorem~\ref{T:bishopsperiod} follows directly.

\medskip
One reason to want the period is a computational method for discovering $u(q;n)$.  To find it (for a fixed number $q$ of pieces) one can count solutions as $n$ ranges from 1 up to some upper limit $N$ and interpolate the counting quasipolynomial from the resulting data.  That can be done if one knows the degree of the quasipolynomial, which is $2q$ by Lemma~I.2.1 (that is, Lemma 2.1 of Part I), and the period, for which there is no known general formula; then $N=2qp$ suffices (since the leading term is $n^{2q}/q!$ by general Ehrhart theory; see Lemma~I.2.1).  Evidently, knowing the period is essential to knowing the right value of $N$, if the formula is to be considered proven.  
In general, for a particular rider piece and number $q$ it is very hard to find the period; its value is known only for trivial pieces or very small values of $q$.  
In contrast, Theorem~\ref{T:bishopsperiod} gives the exact period for bishops, and it follows that to find the exact number of placements of $q$ bishops it suffices to compute only $4q$ values of the counting function.

The reader may ask why we do not seek the complete formula for bishops placements in terms of both $n$ and $q$.  
Remarkably, there is a simple such formula, due in essence to Arshon in a nearly forgotten paper \cite{Arshon} and completed by \Kot\ \cite[6th ed., pp.\ 244, 254--257]{ChMath}.  We restated this expression in Section~\B\ (that is, Section~\BB\ of Part V \cite{QQs4b}).  The trouble is that it is not in the form of a quasipolynomial, so that, for instance, we could not use it to obtain the number of combinatorial types of nonattacking configuration, which by Theorem~I.5.3 is its evaluation at $n=-1$.  We cannot even deduce the period from the Arshon formula.%
\footnote{Stanley in \cite[Solution to Exercise 4.42]{EC1-2} says the period is easily obtained from Arshon's formula, which has one form for even $n$ and another for odd $n$; but we think it is not that easy.}
So there is reason to seek the general quasipolynomial $q!u_\pB(q;n)$ for every number $q$.  
The simple reason we do not seek to do so is that we have not found a way to do it.  That remains an open problem whose solution would reveal the full character of the dependence of $u_\pB(q;n)$ on its two arguments.  
This has not yet been discovered for any rider---other than the mathematically trivial rook.

\medskip
After necessary mathematical background in the next two sections, we prove Theorem~\ref{T:bishopsperiod} in Section~\ref{Bperiod}, applying the geometry of the inside-out polytope for bishops and the properties of signed graphs, which we introduce in Sections~\ref{partI} and \ref{sg}, respectively.  
We conclude with research questions.  For the benefit of the authors and readers, we append a dictionary of the notation in this paper.


\section{Essentials from Parts I and II}\label{partI}  

We build upon the counting theory of previous parts as it applies to the square board, from Part~II \cite{QQs2}.  We summarize essential aspects here.  First, we specialize our notation to $q$ nonattacking bishops $\pB$\label{d:B} on a square board.  We assume that $q > 0$.

The full expression for the number of nonattacking configurations of unlabelled bishops is 
\begin{equation*}
u_\pB(q;n) = \gamma_{0}(n) n^{2q} + \gamma_{1}(n) n^{2q-1} + \gamma_{2}(n) n^{2q-2} + \cdots + \gamma_{2q}(n) n^0,
\label{d:gamma}
\end{equation*}
where each coefficient $\gamma_i(n)$ varies periodically with $n$, and for labelled pieces the number is $o_\pB(q;n)$\label{d:distattacks}, which equals $q!u_\pB(q;n)$.  (The coefficients also depend on $q$ but we suppress that in the notation because only the variation with $n$ concerns us here.)

The $n \times n$ board consists of the integral points in the interior $(n+1) (0,1)^2$ of an integral multiple $(n+1) [0,1]^2$ of the unit square $\cB = [0,1]^2 \subset \bbR^2$,\label{d:Board} or equivalently, the $1/(n+1)$-fractional points in $(0,1)^2$.  Thus, the board consists of the points $z=(x,y)$ for integers $x,y = 1,2,\ldots,n$.  

A \emph{move} is the difference between a new position and the original position.  
The bishop has moves given by all integral multiples of the vectors $(1,1)$ and $(1,-1)$, which are called the \emph{basic moves}.\label{d:mr}  
(Note that for a move $m=(c,d)$, the slope $d/c$ contains all necessary information and can be specified instead of $m$ itself.)  
A bishop in position $z=(x,y)$\label{d:z} may move to any location $z + \mu m$ with $\mu \in \bbZ$ and a basic move $m$, provided that location is on the board.  

A \emph{configuration} is the vector $(z_1,z_2,\ldots,z_q)$ of positions of all $q$ bishops.  The constraint on a configuration is that no two pieces may attack each other, or to say it mathematically, when there are pieces at positions $z_i$ and $z_j$, then $z_j-z_i$ is not a multiple of any basic move $m$.  

The object on which our theory relies is the \emph{inside-out polytope} $(\cP,\cA_\pB)$, where $\cP$ is the hypercube $\cube$ and $\cA_\pB$ is the \emph{move arrangement}\label{d:AP} for bishops.  The move arrangement is a finite set of hyperplanes whose members are the \emph{move hyperplanes} or \emph{attack hyperplanes},
$$
\cH^\pm_{ij} := \{ \bz \in \bbR^{2q} : (y_j-y_i)=\pm(x_j-x_i) \}.
\label{d:slope-hyp}
$$ 
Each attack hyperplane contains the configuration points $z = (z_1,z_2,\ldots,z_q) \in \bbZ^{2q}$ in which bishops $i$ and $j$ attack each other.  (The pieces in a configuration are labelled 1 through $q$ to enable effective description.)  
The \emph{intersection lattice} of $\cA_\pB$ is the set of all intersections of subsets of the move arrangement, ordered by reverse inclusion.  These intersection subspaces are the heart of our method.


\section{Signed Graphs }\label{sg}

The signed graph we employ to describe an intersection subspace efficiently is a special case of the slope graph from Section~I.3.3.  The fact that the bishops' two slopes are $\pm1$ makes it possible to apply the well-developed theory of signed graphs.

A graph is $\G = (N,E)$,\label{d:graph} with node set $N$ and edge set $E$.  It may have multiple edges but not loops.  A \emph{$1$-forest} is a graph in which each component consists of a tree together with one more edge; thus, each component contains exactly one circle.  A spanning 1-forest is a spanning subgraph (it contains all nodes) that is a 1-forest.  The notation $e_{ij}$ means the edge has end nodes $v_i$ and $v_j$.

A \emph{signed graph}, $\S = (N,E,\sigma)$,\label{d:sg} is a graph in which each edge $e$\label{d:edge} is labelled $\sigma(e) = +$ or $-$.\label{d:sigma}  In a signed graph, a circle (cycle, circuit) is called \emph{positive} or \emph{negative} according to the product of its edge signs.  A \emph{signed circuit} is either a positive circle or a connected subgraph that contains exactly two circles, both negative.  A node $v$\label{d:node} is \emph{homogeneous} if all incident edges have the same sign.  We generally write $q := |N|$\label{d:qS} because the nodes correspond to the bishops in a configuration.

Let $c(\S)$ denote the number of components of a signed (or unsigned) graph and $\xi(\S) := |E|-|N|+c(\S)$,\label{d:cyclo} the cyclomatic number of the underlying unsigned graph.  

The \emph{incidence matrix} of $\S$ is the $|N| \times |E|$ matrix $\Eta(\S)$\label{d:Eta} ($\Eta$ is ``Eta'') such that, in the column indexed by edge $e$, the elements are $\eta(v,e) = \pm1$ if $v$ is an endpoint of $e$ and $= 0$ if it is not, with the signs chosen so that, if $v_i$ and $v_j$ are the endpoints, then $\eta(v_i,e)\eta(v_j,e) = -\sigma(e)$ \cite[Section~8A]{SG}.  That is, in the column of a positive edge there are one $+1$ and one $-1$, while in the column of a negative edge there are two $+1$'s or two $-1$'s.  
The \emph{rank} of $\S$ is the rank of its incidence matrix.\label{d:rk}  From \cite[Theorem~5.1(j)]{SG} we know a formula for the rank: $\rk(\S) = |N| - b(\S)$, where $b(\S)$\label{d:bSigma} 
is the number of components in which there is no negative circle.  
This rank function applied to spanning subgraphs makes a matroid $G(\S)$ on the edge set of $\S$ \cite{SG}.  
An unsigned graph $\G$ acts as if it is an all-positive signed graph; therefore its incidence matrix has rank $\rk(\G) = |N| - c(\G)$ where $c(\G)$\label{d:comp} 
is the number of components and the corresponding matroid $G(\G) := G(+\G)$ is the cycle matroid of $\G$.  

From this and \cite[Theorem~8B.1]{SG} we also know that $\Eta(\S)$ has full column rank if and only if $\S$ contains no signed circuit and it has full row rank if and only if every component of $\S$ contains a negative circle.  A signed graph that has both of these properties is necessarily a 1-forest in which every circle is negative.  

\label{d:clique}
A \emph{positive clique} in $\S$ is a maximal set of nodes that are connected by positive edges; equivalently, it is the node set of a connected component of the spanning subgraph $\S^+$ formed by the positive edges.  A \emph{negative clique} is similar.  Either kind of set is called a \emph{signed clique}.  We call them ``cliques'' (in a slight abuse of terminology) because the signed cliques of a graph do not change if we complete the induced positive subgraph on a positive clique, and similarly for a negative clique.  
A homogeneous node $v$ gives rise to a singleton signed clique with the sign not represented by an edge at $v$; if $v$ is isolated it gives rise to two singleton cliques, one of each sign.

\label{d:posnegcn}
The number of positive cliques in $\S$ is $c(\S^+)$ and the number of negative cliques is $c(\S^-)$.  
Let $\Alpha(\S) := \{A_1,\ldots,A_{c(\S^+)}\}$ and $\Beta(\S) := \{B_1,\ldots,B_{c(\S^-)}\}$ (read ``Alpha'' and ``Beta'') be the sets of positive and negative cliques, respectively.  Since each node of $\S$ is in precisely one positive and one negative clique, we can define a bipartite graph $C(\S)$,\label{d:CS} called the \emph{clique graph} of $\S$, whose node set is $\Alpha(\S) \cup \Beta(\S)$ and whose edge set is $N$, the endpoints of the edge $v_i$ being the cliques $A \in \Alpha(\S)$ and $B \in \Beta(\S)$ such that $v_i \in A \cap B$.

Let us call an edge $e$ \emph{redundant} if $\S \setm e$ ($\S$ with $e$ deleted) has the same signed cliques as does $\S$, and call $\S$ \emph{irredundant} if it has no redundant edges, in other words, if each signed clique has just enough edges of its sign to connect its nodes.  
A signed graph is irredundant if and only if both $\S^+$ and $\S^-$ are forests.  For example, a signed forest is irredundant.  
Any signed graph can be reduced to irredundancy with the same signed cliques by pruning redundant edges one by one.

\begin{lem}\label{L:sclique}
If $\S$ is a signed graph with $q$ nodes, then $|\Alpha(\S)|+|\Beta(\S)| = 2q - [\rk(\S^+) + \rk(\S^-)]$.  
If $\S$ is irredundant, then $|\Alpha(\S)|+|\Beta(\S)| = 2q - |E| = q + c(\S) - \xi(\S).$  In particular, a signed tree has $q+1$ signed cliques.
\end{lem}

\begin{proof}
The first formula follows directly from the general formula for the rank of a graph.  

If $\S$ is irredundant, $\S^+$ is a forest with $|\Alpha(\S)|$ components and $\S^-$ is a forest with $|\Beta(\S)|$ components.  Therefore, $|\Alpha(\S)| + |\Beta(\S)| = 2q - |E| = q - \xi(\S) + c(\S)$.

A more entertaining proof is by induction on the number of inhomogeneous nodes.  Define $g(\S) := |\Alpha(\S)|+|\Beta(\S)| - 2q + |E| = |\Alpha(\S)|+|\Beta(\S)| - q - c(\S) + \xi(\S).$  If all nodes are homogeneous, obviously $g(\S) = 0$.  Otherwise, let $v$ be an inhomogeneous node.  Split $v$ into two nodes, $v^+$ and $v^-$, incident respectively to all the positive or negative edges at $v$.  The new graph has one less inhomogeneous node, two more signed cliques (a positive clique $\{v^-\}$ and a negative clique $\{v^+\}$), one more node, and the same number of edges, hence the same value of $g$ as does $\S$.  Thus, by induction, $g \equiv 0$.
\end{proof}


\section{Proof of the Bishops Period}\label{Bperiod}

We are now prepared to prove Theorem~\ref{T:bishopsperiod}.  We already proved in Theorem~\Phvdiag\ that the coefficients $\gamma_i(n)$ are constant (as functions of $n$) for $i<6$ and that $\gamma_6(n)$ has period 2.  Thus it will suffice to prove that the denominator of the inside-out polytope $(\cB,\cA_\pB)$ for $q$ bishops divides 2.  
(In fact, what we prove is the stronger result stated in Lemma~\ref{L:bishopsdenom}.) 
To do this, we find the denominators of all vertices explicitly by analyzing all sets of $2q$ equations that determine a point.  We use the polytope $\cube$ for the boundary inequalities and the move arrangement $\cA_\pB$ for the equations of attack.   

We use a fundamental fact from linear algebra.

\begin{lem}\label{L:vertexgeo}
The coordinates $z_i=(x_i,y_i)$ belong to a vertex of the inside-out polytope if and only if there are $k$ attack equations and $2q-k$ boundary equations that uniquely determine those coordinates.
\end{lem}

We assume the $q$ bishops are labelled $\pB_1,\ldots,\pB_q$.  
A configuration of bishops is described by a point $\bz = (z_1,z_2,\ldots,z_q) \in \bbR^{2q}$,\label{d:bfz} where $z_i = (x_i,y_i)$ is the normalized plane coordinate vector of the $i$th bishop $\pB_i$; that is, $x_i,y_i \in (0,1)$ and the position of $\pB_i$ is $(n+1)z_i$.  The bishops constraints are that $\bz$ should not lie in any of the $q(q-1)$ \emph{bishops hyperplanes},
\begin{align}
\cH_{ij}^+: x_i-y_i = x_j-y_j, \qquad \cH_{ij}^-: x_i+y_i = x_j+y_j,
\label{E:bishopseqns}
\end{align}
where $i\neq j$.  The corresponding equations are the \emph{bishops equations} and a subspace $\cU$ defined by a set of bishops equations is a \emph{bishops subspace}.\label{d:cU} 
The boundary equations of $\cube$ have the form $x_i=0$ or $1$ and $y_i=0$ or $1$.  
We generalize the boundary constraints; we call any equation of the form $x_i=c_i \in \bbZ$ or $y_i=d_i \in \bbZ$ a \emph{fixation}.\label{d:fixnconst}
We call any point of $\bbR^{2q}$ determined by $m$ bishops equations and $2q-m$ fixations a \emph{lattice vertex}.  
(The term ``fixation'' was used in Part~V only for a boundary equation; here it means any equation that fixes one coordinate at an integral value.) 

The first step is to find the dimension of a bishops subspace.  
We do so by means of a signed graph $\S_\pB$ with node set $N := \{v_1,v_2,\ldots,v_q\}$ corresponding to the bishops $\pB_i$ and their plane coordinates $z_i=(x_i,y_i)$ and with edges corresponding to the bishops hyperplanes.  
For a hyperplane $\cH_{ij}^+$ we have a \emph{positive edge} $e_{ij}^+$\label{d:sgdedge} and for a hyperplane $\cH_{ij}^-$ we have a \emph{negative edge} $e_{ij}^-$.  Thus, $\S_\pB$ is a complete signed link graph: it has all possible edges (barring loops, of which we have no need) of both signs.  
For each bishops subspace $\cU$ we have a spanning subgraph $\S(\cU)$\label{d:bishopsgraph}
whose edges correspond to the bishops hyperplanes that contain $\cU$.  (This is nothing other than the slope graph defined in Section~I.3.3, except that it has extra nodes to make up a total of $q$.)  Then $\cU$ is the intersection of all the hyperplanes whose corresponding edges are in $\S(\cU)$.

\begin{lem}\label{L:bishopdim}
For any $\cS \subseteq \cA_\pB$,\label{d:cS} with corresponding signed graph $\S \subseteq \S_\pB$, $\codim\bigcap\cS = \rk(\S^+) + \rk(\S^-)$.  
For a bishops subspace $\cU$, $\dim \cU = |\Alpha(\S(\cU))| + |\Beta(\S(\cU))|$ and $\codim \cU = \rk(\S(\cU)^+) + \rk(\S(\cU)^-)$.  
\end{lem}

\begin{proof}
We begin with $\cS$ by looking at a single sign.  Adjacent edges $e_{ij}^\epsilon, e_{jk}^\epsilon$ of sign $\epsilon$\label{d:edgesign} in $\S$, corresponding to $\cH_{ij}^\epsilon$ and $\cH_{jk}^\epsilon$, imply the third positive edge because the hyperplanes' equations imply that of $\cH_{ik}^\epsilon$.  Consequently we may replace $E(\S^\epsilon)$ by a spanning tree of each $\epsilon$-signed clique without changing the intersection subspace.  Call the revised graph $\S'$.  
Being irredundant, it has $2q - (|\Alpha(\S)|+|\Beta(\S)|)$ edges by Lemma~\ref{L:sclique}.  As each hyperplane reduces the dimension of the intersection by at most 1, we conclude that $\codim\bigcap\cS \leq 2q - (|\Alpha(\S)|+|\Beta(\S)|)$.

On the other hand it is clear that $\cA_\pB$ intersects in the subspace $\{(z,z,\ldots,z): z \in \bbR^2\}$; thus, $2q-2 = \codim\bigcap\cA_\pB$.  The corresponding signed graph $\S_\pB$, when reduced to irredundancy, consists of a spanning tree of each sign; in other words, it has $2q-2$ edges.  One can choose the irredundant reduction of $\S_\pB$ to contain $\S'$; it follows that every hyperplane of $\cS$ must reduce the dimension of the intersection by exactly 1 in order for the reduced $\S_\pB$ to correspond to a 2-dimensional subspace of $\bbR^2$.  Therefore, $\codim\bigcap\cS = |E(\S')| = 2q - (|\Alpha(\S)| + |\Beta(\S)|) = \rk(\S^+) + \rk(\S^-)$.

The dimension formula for $\cU$ follows by taking $\cS := \{H \in \cA_\pB : H \supseteq \cU\}$.  
\end{proof}

Defining the rank of an arrangement $\cA$ of hyperplanes to be the codimension of its intersection yields a matroid whose ground set is $\cA$.  The matroid's rank function encodes the linear dependence structure of the bishops arrangement $\cA_\pB$.  The complete graph of order $q$ is $K_q$.\label{d:K}

\begin{prop}\label{P:bishopmatroid}
The matroid of the hyperplane arrangement $\cA_\pB$ is isomorphic to $G(K_q) \oplus G(K_q)$.  
\end{prop}

\begin{proof}
The rank of $\cS \subseteq \cA_\pB$, corresponding to $\S \subseteq \S_\pB$, is the codimension of $\bigcap \cS$, which by Lemma~\ref{L:bishopdim} equals $\rk(\S^+) + \rk(\S^-)$.  The matroid this implies on $E(\S_\pB)$ is the direct sum of $G(\S_\pB^+)$ and $G(\S_\pB^-)$.  Both $\S_\pB^+$ and $\S_\pB^-$ are unsigned complete graphs.  The proposition follows.
\end{proof}

Now we return to the analysis of a lattice vertex $\bz$.  
A point is \emph{strictly half integral} if its coordinates have least common denominator 2; it is \emph{weakly half integral} if its coordinates have least common denominator 1 or 2.  A \emph{weak half integer} is an element of $\frac12\bbZ$; a \emph{strict half integer} is a fraction that, in lowest terms, has denominator 2.  

\begin{lem}\label{L:bishopsdenom}
A point $\bz = (z_1,z_2,\ldots,z_q) \in \bbR^{2q}$, determined by a total of $2q$ bishops equations and fixations, is weakly half integral.  Furthermore, in each $z_i$, either both coordinates are integers or both are strict half integers.
\end{lem}

Consequently, a vertex of the bishops' inside-out polytope $(\cube,\cA_\pB)$ has each $z_i \in \{0,1\}^2$ or $z_i = (\frac12,\frac12)$.

\begin{proof}
For the lattice vertex $\bz$, find a bishops subspace $\cU$ such that $\bz$ is determined by membership in $\cU$ together with $\dim \cU$ fixations.  

Suppose $v_i,v_j \in A_k$,\label{d:ABkl} a positive clique in $\S(\cU)$; then $x_i - y_i = x_j - y_j$; thus, the value of $x_i - y_i$ is a constant $a_k$ on $A_k$.  Similarly, $x_i + y_i$ is a constant $b_l$ on each negative clique $B_l$.

Now replace $\S(\cU)$ by an irredundant subgraph $\S$ with the same positive and negative cliques.  The edges of $\S$ within each clique are a tree.  The total number of edges is $2q-(|\Alpha(\S(\cU))|+|\Beta(\S(\cU))|)$; this is the number of bishops equations in the set determining $\bz$.  The remaining $|\Alpha(\S(\cU))|+|\Beta(\S(\cU))|$ equations are fixations.

Write $C_\cU$\label{d:cliq} for the clique graph $C(\S) = C(\S(\cU))$.  Let $\mp C_\cU$ be the graph $C_\cU$ with each edge $v_i$ replaced by two edges called $v_i^x$ and $v_i^y$.  If we (arbitrarily) regard $x$ as $-$ and $y$ as $+$, this is a signed graph.

We defined $a_k$ and $b_l$ in terms of the $x_i$ and $y_i$.  We now reverse the viewpoint, treating the $a$'s and $b$'s as independent variables and the $x$'s and $y$'s as dependent variables.  This is possible because, if $A_k, B_l$ are the endpoints of $v_i$ in $C_\cU$, then 
$x_i - y_i = a_k \text{ and } x_i + y_i = b_l,$
so
$$
x_i = \frac12(a_k+b_l) \quad \text{ and } \quad y_i = \frac12(-a_k+b_l);
$$
in matrix form,
\begin{equation}\label{E:zfromab}
\begin{bmatrix} \bx \\ \by \end{bmatrix} 
= \frac12 \begin{bmatrix} \Eta(-C_\cU)\transpose \\ \Eta(+C_\cU)\transpose \end{bmatrix} \begin{bmatrix} \ba \\ \bfb \end{bmatrix} 
= \frac12 \Eta(\mp C_\cU)\transpose \begin{bmatrix} \ba \\ \bfb \end{bmatrix} ,
\end{equation}
where $\bx = \begin{bmatrix} x_i \end{bmatrix}_{i=1}^q$,\label{bf x} $\by = \begin{bmatrix} y_j \end{bmatrix}_{j=1}^q$, $\ba = \begin{bmatrix} a_k \end{bmatrix}_{k=1}^{|\Alpha(\S(\cU))|}$,\label{bf a} and $\bfb = \begin{bmatrix} b_l \end{bmatrix}_{l=1}^{|\Beta(\S(\cU))|}$ are column vectors and $\Eta(\epsilon C_\cU)$ is the incidence matrix of $C_\cU$ with, respectively, all edges positive for $\epsilon=+$ (with $-$ and $+$ at the $\Alpha$ and $\Beta $ ends) and all edges negative for $\epsilon=-$ (with $+$ at both ends).  Thus, the first coefficient matrix is the transposed incidence matrix of $\mp C_\cU$ written with a particular ordering of the edges.  

Fixing a total of $|\Alpha(\S(\cU))|+|\Beta(\S(\cU))|$ variables $x_{i_1},\ldots$ and $y_{j_1},\ldots$ should determine all the values $x_1, y_1, \ldots, x_q, y_q$.  The fixations of $\bz$ correspond to edges in $\mp C_\cU$ so we may treat a choice of fixations as a choice of edges of $\mp C_\cU$, where fixing $x_i$ or $y_i$ corresponds to choosing the edge $v_i^x$ or $v_i^y$.  We need to know what kind of edge set the fixations correspond to.  
Let $\Ps_\bz$\label{d:Psiz} denote the spanning subgraph of $\mp C_\cU$ whose edges are the chosen edges.  The fixation equations can be written in matrix form as 
\begin{equation}\label{E:fixation-equations}
M\transpose \begin{bmatrix} \ba \\ \bfb \end{bmatrix}
=
2 \begin{bmatrix} x_{i_1} \\ \vdots \\ y_{j_1} \\ \vdots \end{bmatrix} 
=
2 \begin{bmatrix} \bc \\ \bd \end{bmatrix} ,
\end{equation}
where the fixation edges are $v_{i_1}^x, \ldots$ with endpoints $A_{k_1},B_{l_1},\ldots$ and $v_{j_1}^y, \ldots$ with endpoints $A_{k'_1},B_{l'_1},\ldots$; the fixations are $x_{i_r}=c_r$ and $y_{j_s}=d_s$; $\bc = \begin{bmatrix} c_r \end{bmatrix}_{r=1}^{\bar r}$\label{bf c} and $\bd = \begin{bmatrix} d_s \end{bmatrix}_{s=1}^{\bar s}$ are column vectors (with $\bar r+\bar s=|\Alpha(\S(\cU))|+|\Beta(\S(\cU))|$, the total number of fixations); and $M$ is an $(|\Alpha(\S(\cU))|+|\Beta(\S(\cU))|) \times (|\Alpha(\S(\cU))|+|\Beta(\S(\cU))|)$ matrix representing the relationships between the $a$'s and $b$'s and the fixed variables:
$$
M := \begin{matrix}
\begin{array}{cccc}
x_{i_1} & \quad & y_{j_1} & \quad \\
\end{array}
&
\\[8pt]
\left[ 
\begin{array}{cccc}
1 & \cdots &  0 & \cdots  \\
\vdots & \ddots & \vdots & \ddots \\
0 & \cdots &  -1 & \cdots \\
\vdots & \ddots & \vdots & \ddots \\[2pt]
\hline \\[-8pt]
1 & \cdots & 0 & \cdots \\
\vdots & \ddots & \vdots & \ddots \\
0 & \cdots & 1 & \cdots \\
\vdots & \ddots & \vdots & \ddots \\
\end{array}
\right]  
&
\begin{matrix}
\\\\ \Alpha(\S(\cU)) \\\\ \\ \\\\ \Beta(\S(\cU)) \\\\ \\
\end{matrix} 
\end{matrix} \ .
$$
The rows of $M$ are indexed by the signed cliques and the columns are indexed by the fixations.  The column of a fixation involving a node $v_i$, whose endpoints in $C_\cU$ are $A_k$ and $B_l$, has exactly two nonzero entries, one in row $A_k$ and one in row $B_l$, whose values are, respectively, $1,1$ for an $x$-fixation and $-1,1$ for a $y$-fixation.  Thus, each column has exactly two nonzero elements, each of which is $\pm1$.  

Consequently, $M$ is the incidence matrix of a signed graph, in fact, $M = \Eta(\Ps_\bz)$.\label{d:M}  $M$ must be nonsingular since the fixed $x$'s and $y$'s uniquely determine the $a$'s and $b$'s (because they determine $\bz$).  It follows (see Section~\ref{sg}) that the fixation equations for $\bz$ are a set corresponding to a spanning 1-forest in $\mp C_\cU$ in which every circle is negative.  This 1-forest is $\Ps_\bz$.  There is choice in the selection of $\Ps_\bz$ but it is not completely arbitrary.\label{d:J}  Let $J_\bz$ be the set of nodes $v_i$ such that $z_i$ is integral; consider $J_\bz$ as a subset of $E(C_\cU)$.  As fixations must be integral, $E(\Ps_\bz)$ must be a subset of $\pm J_\bz$.  As fixations are arbitrary integers, $\Ps_\bz$ may be any spanning 1-forest of $\mp C_\cU$ that is contained in $\pm J_\bz$ and whose circles are negative.  Thus we have found the graphical form of the equations of a lattice vertex.  

\begin{exam}\label{X:z}
For an example, suppose there are three positive and four negative cliques, so $\Alpha(\S(\cU)) = \{A_1,A_2,A_3\}$ and $\Beta(\S(\cU)) = \{B_1,B_2,B_3,B_4\}$, and eight nodes, $N = \{v_1,\ldots,v_8\}$, with the following clique graph $C_\cU$:
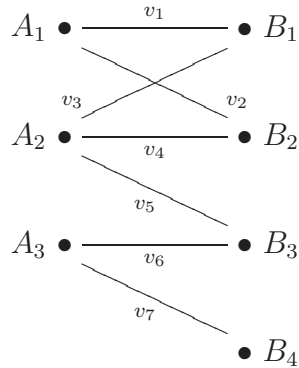
$$
\xymatrix{
A_1 \ \bullet \ar@{-}[rr]^{v_1} \ar@{-}[rrd]^>{v_2} 
&& \bullet \ B_1 \\
A_2 \ \bullet \ar@{-}[rru]^<{v_3} \ar@{-}[rr]_{v_4} \ar@{-}[rrd]_{v_5} 
&& \bullet \ B_2 \\
A_3 \ \bullet \ar@{-}[rr]_{v_6} \ar@{-}[rrd]_{v_7} 
&& \bullet \ B_3 \\
&& \bullet \ B_4 \\
}
$$
An example of a suitable 1-forest $\Ps_\bz \subseteq \mp C_\cU$ is
$$
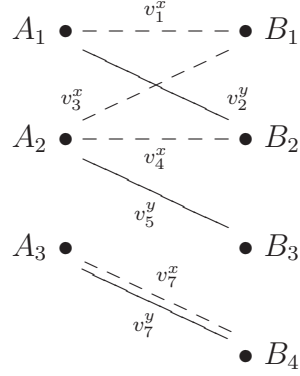
\xymatrix{
A_1 \ \bullet \ar@{--}[rr]^{v_1^x} \ar@{-}[rrd]^>{v_2^y} 
&& \bullet \ B_1 \\
A_2 \ \bullet \ar@{--}[rru]^<{v_3^x} \ar@{--}[rr]_{v_4^x} \ar@{-}[rrd]_{v_5^y} 
&& \bullet \ B_2 \\
A_3 \ \bullet \ar@<.4ex>@{--}[rrd]^{v_7^x} \ar@<-.2ex>@{-}[rrd]_{v_7^y} 
&& \bullet \ B_3 \\
&& \bullet \ B_4 \\
}
$$
It corresponds to fixations 
$$x_1=c_1,\ y_2=d_1,\ x_3=c_2,\ x_4=c_3,\ y_5=d_2,\ x_7=c_4,\ y_7=d_3.$$  
The incidence matrix is
$$
M := \Eta(\Ps_\bz) = \begin{array}{cc}
\begin{array}{ccccccc}
\ x_1 & x_3 &\ x_4 &\ x_7 &\ y_2 &\ y_5 &\ y_7 \\
\end{array}
&
\\[8pt]
\left[ 
\begin{array}{ccccccc}
\ 1\ &\ 0\ &\ 0\ &\ 0\ & -1 & 0 & 0  \\
0 & 1 & 1 & 0 & 0 & -1 & 0 \\
0 & 0 & 0 & 1 & 0 & 0 & -1 \\[2pt]
\hline \\[-8pt]
1 & 1 & 0 & 0 & 0 & 0 & 0 \\
0 & 0 & 1 & 0 & 1 & 0 & 0 \\
0 & 0 & 0 & 0 & 0 & 1 & 0 \\
0 & 0 & 0 & 1 & 0 & 0 & 1 \\
\end{array}
\right]  
&
\begin{matrix}
A_1 \\ A_2 \\ A_3 \\[7pt] B_1 \\ B_2 \\ B_3 \\ B_4 \\
\end{matrix} 
\end{array}.
$$
Every column has two nonzeros.  The equations of the fixations in matrix form are
$$
M\transpose
\begin{bmatrix}
a_1 \\ a_2 \\ a_3 \\ b_1 \\ b_2 \\ b_3 \\ b_4 \\
\end{bmatrix}
=
2 \begin{bmatrix}
x_1 \\ x_3 \\ x_4 \\ x_7 \\ y_2 \\ y_ 5 \\ y_7 \\
\end{bmatrix}
=
2 \begin{bmatrix}
c_1 \\ c_2 \\ c_3 \\ c_4 \\ d_1 \\ d_2 \\ d_3 \\
\end{bmatrix} ,
$$
where the $c_i$'s and $d_j$'s are any integers we wish in the lemma (but in the application to Theorem~\ref{T:bishopsperiod} they will be 0's and 1's).  The solution is 
\begin{gather*}
a_1=x_1-x_3+x_4-y_2=c_1-c_2+c_3-d_1, \\
a_2=-x_1+x_3+x_4-y_2=-c_1+c_2+c_3-d_1, \\
a_3=x_7-y_7=c_4-d_3, \\
b_1=x_1+x_3-x_4+y_2=c_1+c_2-c_3+d_1, \\
b_2=x_1-x_3+x_4+y_2=c_1-c_2+c_3+d_1, \\
b_3=-x_1+x_3+x_4-y_2+2y_5=-c_1+c_2+c_3-d_1+2d_2, \\
b_4=x_7+y_7=c_4+d_3, 
\end{gather*}
and the unfixed variables are
\begin{gather*}
x_2=\frac{a_1+b_2}{2} = c_1-c_2+c_3, \\
x_5=\frac{a_2+b_3}{2} = -c_1+c_2+c_3-d_1+d_2, \\
x_6=\frac{a_3+b_3}{2} = \frac{-c_1+c_2+c_3+c_4-d_1+2d_2-d_3}{2}, \\
y_1=\frac{-a_1+b_1}{2} = c_2-c_3+d_1, \\
y_3=\frac{-a_2+b_1}{2} = c_1-c_3+d_1, \\
y_4=\frac{-a_2+b_2}{2} = c_1-c_2+d_1, \\
y_6=\frac{-a_3+b_3}{2} =\frac{-c_1+c_2+c_3-c_4-d_1+2d_2+d_3}{2}.
\end{gather*}

Observe that $x_6$ and $y_6$ are the only possibly fractional coordinates and their difference, $x_6-y_6 = a_3=c_4-d_3$, is integral; therefore, either $z_6$ is integral, or both $x_6$ and $y_6$ are half integers and $z_6=(\frac12,\frac12)$ if $\bz\in\cube$.
\end{exam}

We are now prepared to prove Lemma~\ref{L:bishopsdenom}.  We need a result from (e.g.) \cite{HMNT}, which can be stated:

\begin{lem}\label{L:halfintegral}
The solution of a linear system with integral constant terms, whose coefficient matrix is the transpose of a nonsingular signed-graph incidence matrix, is weakly half-integral.
\end{lem}

\begin{proof}
The way in which this is contained in \cite{HMNT} is explained in \cite[p.\ 197]{AKbi}.  
\end{proof}

Since $M$ is the incidence matrix of a signed graph, and since the constant terms in Equation~\eqref{E:fixation-equations},  being twice the fixed values, are even integers, the $a$'s and $b$'s are integers by Lemma~\ref{L:halfintegral}.  The remaining $x$'s and $y$'s are halves of sums or differences of $a$'s and $b$'s, so they are weak half-integers.  The exact formula is obtained by substituting Equation~\eqref{E:fixation-equations} into Equation~\eqref{E:zfromab}:
\begin{equation}\label{E:determination}
\begin{bmatrix} \bx \\ \by \end{bmatrix} 
= \Eta(\mp C_\cU)\transpose (M\inv)\transpose \begin{bmatrix} \bc \\ \bd \end{bmatrix}.
\qedhere
\end{equation}
\end{proof}

Theorem~\ref{T:bishopsperiod} is an immediate corollary of Lemma~\ref{L:bishopsdenom}.

\section{Open Questions}\label{questions}

\subsection{Coefficient periods}\label{allcoeffp}\

We proved that $\gamma_6(n)$ is the first coefficient that depends on $n$, having period 2.  We guess that every coefficient after $\gamma_6(n)$ also has period 2.

\subsection{Subspace structure}\label{allsubspaces}\

We have not been able to find a complete formula for all $q$.  By our method, that would need a general structural analysis of all subspaces, which is too complicated for now.  We propose the following problem:  Give a complete description of all subspaces, for all $q$, in terms of signed graphs.  That is, we ask for the slope matroid (see Section~I.7.3).  The signed-graphic frame matroid $G(\S)$ (\cite[Theorem 5.1]{SG}, corrected and generalized in \cite[Theorem 2.1]{BG2}), while simpler than the slope matroid, perhaps could help find a description of the latter.

\subsection{Similar two-move riders}\label{similar}\

The slope matroid for the bishop is simple compared to those for other riders.  We wonder if riders with two slopes that are related by negation (that is, the basic moves are symmetrical under reflection in an axis), or negation and inversion (that is, the basic moves are perpendicular), may be amenable to an analysis that uses the bishops analysis as a guide.

\subsection{Other two-move riders}\label{other}\

We expect that finding formulas for any rider with only two basic moves is intrinsically easier than for riders with more than two and can be done for all such riders in a comprehensive though complicated manner.

\section*{Dictionary of Notation}

\begin{enumerate}[]
\item $b(\S)$ -- \# of signed-graph components with no negative circle (p.\ \pageref{d:bSigma})	
\item $c(\G)$, $c(\S)$ -- \# of components of a graph (p.\ \pageref{d:comp})	
\item $c(\S^\pm)$ -- \# of positive or negative cliques (p.\ \pageref{d:posnegcn})	
\item $d/c$ -- slope of a line or move (p.\ \pageref{d:slope-hyp})
\item $(c,d)$ -- coordinates of a move vector $m$ (p.\ \pageref{d:mr})
\item $c_i,d_i$ -- fixation equation constants (p.\ \pageref{d:fixnconst})
\item $e$ -- edge of a (signed) graph (p.\ \pageref{d:edge})	
\item $e_{ij}^\epsilon$ -- edge of a signed graph with sign $\epsilon$ and end nodes $v_i, v_j$ (p.\ \pageref{d:sgdedge})	
\item $g(\S)$ -- function on a signed graph (p.\ \pageref{L:sclique})	
\item $k,l$ -- indices in the clique graph (p.\ \pageref{d:ABkl})	
\item $m = (c,d)$ -- basic move (p.\ \pageref{d:mr})
\item $n$  -- size of a square board
\item $o_\pB(q;n)$ -- \# of nonattacking labelled configurations (p.\ \pageref{d:distattacks})
\item $p$ -- period of a quasipolynomial (p.\ \pageref{d:p})
\item $q$ -- \# of pieces on a board (p.\ \pageref{d:q})
\item $q$ -- \# of nodes in a (signed) graph (p.\ \pageref{d:qS})	
\item $r,s$ -- indices of fixations (p.\ \pageref{E:fixation-equations})	
\item $u_\pB(q;n)$ -- \# of nonattacking unlabelled configurations (p.\ \pageref{d:indistattacks})
\item $v$ -- node in a signed graph (p.\ \pageref{d:node})	
\item $z=(x,y)$, $z_i=(x_i,y_i)$ -- piece positions (p.\ \pageref{d:z})
\end{enumerate}
\smallskip
\begin{enumerate}[]
\item $\ba$, $\bfb$ -- clique vectors (p.\ \pageref{bf a})
\item $\bc$, $\bd$ -- fixation vectors (p.\ \pageref{bf c})
\item $\bx$, $\by$ -- $x$, $y$ coordinate vectors of a configuration (p.\ \pageref{bf x})
\item $\bz = (z_1,\ldots,z_q)$ -- a configuration in $\bbR^{2q}$ (p.\ \pageref{d:bfz})
\end{enumerate}
\smallskip
\begin{enumerate}[]
\item $\gamma_{i}(n)$ -- coefficient of $u_\pB$ (p.\ \pageref{d:gamma})
\item $\epsilon$ -- sign of an edge (p.\ \pageref{d:edgesign})
\item $\xi$ -- cyclomatic number (p.\ \pageref{d:cyclo})	
\item $\sigma$ -- sign function of the signed graph $\Sigma$ (p.\ \pageref{d:sigma})	
\end{enumerate}
\bigskip
\begin{enumerate}[]
\item $\rk$ -- rank of the incidence matrix of a (signed) graph (p.\ \pageref{d:rk})
\end{enumerate}
\bigskip
%
\begin{enumerate}[]
\item $A_k,B_l$ -- positive, negative cliques (p.\ \pageref{d:clique})	
\item $C(\S)$ -- clique graph (p.\ \pageref{d:CS})	
\item $C_\cU = C(\S(\cU))$ -- clique graph (p.\ \pageref{d:cliq})	
\item $E$ -- edge set of a graph (p.\ \pageref{d:graph})	
\item $G$ -- matroid on ground set $E$ (p.\ \pageref{d:bSigma})	
\item $J_\bz$ -- set of vertices $z_i$ in the configuration $\bz$ such that $z_i$ is integral  (p.\ \pageref{d:J}) 
\item $K_q$ -- complete graph (p.\ \pageref{d:K})
\item $M$ -- incidence matrix $\Eta(\Ps_\bz)$ (p.\ \pageref{d:M}) 	
\item $N$ -- node set of a graph (p.\ \pageref{d:graph})
\end{enumerate}
\bigskip
\begin{enumerate}[]
\item $\cA_{\pB}$ -- move arrangement of bishops $\pB$ (p.\ \pageref{d:AP})
\item $\cB$ -- board polygon $[0,1]^q$ (p.\ \pageref{d:Board})
\item $\cH_{ij}^\pm$ -- bishops hyperplane (p.\ \pageref{T:bishopsperiod})	
\item $(\cP,\cA)$ -- inside-out polytope (p.\ \pageref{d:AP})
\item $\cS$ -- subarrangement  (p.\ \pageref{d:cS})
\item $\cU$ -- subspace in the intersection lattice of an arrangement (p.\ \pageref{d:cU}) 
\end{enumerate}
\bigskip
\begin{enumerate}[]
\item $\bbR$ -- real numbers
\item $\bbZ$ -- integers
\end{enumerate}
\bigskip
\begin{enumerate}[]
\item $\pB$ -- bishop (p.\ \pageref{d:B})
\end{enumerate}
\bigskip
\begin{enumerate}[]
\item $\Alpha(\S), \Beta(\S)$ -- sets of positive, negative cliques (p.\ \pageref{d:clique})	
\item $\G$ -- graph (p.\ \pageref{d:graph})	
\item $\Eta$ -- incidence matrix (read ``Eta'') of a (signed) graph (p.\ \pageref{d:Eta})	
\item $\S$ -- signed graph (p.\ \pageref{d:sg})	
\item $\S(\cU)$ -- signed graph of the bishops subspace $\cU$ (p.\ \pageref{d:bishopsgraph})	
\item $\Ps_\bz$ -- subgraph for a vertex $\bz$ (p.\ \pageref{d:Psiz})	
\end{enumerate}

\newpage

\newcommand\otopu{$\overset{\circ}{\textrm u}$}

\end{document}